\documentclass[leqno]{amsart}
\usepackage{times}
\usepackage{amsfonts,amssymb,amsmath,amsgen,amsthm}
\usepackage{hyperref}
\usepackage{color}
\newcommand{\msc}[2][2000]{%
  \let\@oldtitle\@title%
  \gdef\@title{\@oldtitle\footnotetext{#1 \emph{Mathematics subject
        classification.} #2}}%
}

\theoremstyle{plain}
\newtheorem{theorem}{Theorem}[section]
\newtheorem{definition}[theorem]{Definition}
\newtheorem{assumption}[theorem]{Assumption}
\newtheorem{lemma}[theorem]{Lemma}
\newtheorem{corollary}[theorem]{Corollary}
\newtheorem{proposition}[theorem]{Proposition}

\theoremstyle{remark}
\newtheorem{remark}[theorem]{Remark}

\def\C{{\mathbb C}}
\def\R{{\mathbb R}}
\def\N{{\mathbb N}}

\def\F{\mathcal F}

\def\({\left(}
\def\){\right)}
\def\<{\left\langle}
\def\>{\right\rangle}
\def\le{\leqslant}
\def\ge{\geqslant}

\def\Tend#1#2{\mathop{\longrightarrow}\limits_{#1\rightarrow#2}}

\def\d{{\partial}}
\def\eps{\varepsilon}
\def\l{\lambda}
\def\om{\omega}
\def\si{{\sigma}}

\DeclareMathOperator{\RE}{Re}
\DeclareMathOperator{\IM}{Im}

\numberwithin{equation}{section}

\begin{document}

\title[Finite time extinction for  NLS]{Finite time extinction for
  nonlinear Schr\"odinger equation in 1D and 2D}

\author[R. Carles]{R\'emi Carles}
\address{CNRS \& Univ. Montpellier~2\\Math\'ematiques
\\CC~051\\34095 Montpellier\\ France}
\email{Remi.Carles@math.cnrs.fr}

\author[T. Ozawa]{Tohru Ozawa}
\address{Department of Applied Physics\\ Waseda University\\ Tokyo
  169-8555\\  Japan}
\email{txozawa@waseda.jp}

\begin{abstract}
  We consider a nonlinear Schr\"odinger equation with power
  nonlinearity, either on a compact manifold without boundary, or on
  the whole space in the presence of harmonic confinement, in space
  dimension one and two. Up to introducing an extra superlinear damping to
  prevent finite time blow up, we show that the presence of a sublinear
  damping always leads to finite time extinction of the solution in
  1D, and that the same phenomenon is present in the case of small mass
  initial data in 2D. 
\end{abstract}
\thanks{RC was supported by the French ANR projects
  SchEq (ANR-12-JS01-0005-01) and BECASIM
  (ANR-12-MONU-0007-04).} 
\maketitle

\section{Introduction}
\label{sec:intro}
In \cite{CaGa11}, the following equation was considered on a compact
manifold without boundary:
\begin{equation*}
i\d_t u +\frac{1}{2}\Delta u = -ib\frac{u}{|u|^\alpha}, \quad t\ge 0,  
\end{equation*}
for $b>0$ and $\alpha\in (0,1]$. This sublinear damping leads to finite time
extinction of the solution, that is $\|u(t)\|_{L^2}=0$ for $t\ge T$,
a phenomenon closely akin to the model involving such a damping is
mechanics \cite{AdAtCa06}. In the one-dimensional case, finite time
extinction was proved for 
\begin{equation*}
i\d_t u +\frac{1}{2}\Delta u = \l |u|^{2\si}u-ib\frac{u}{|u|^\alpha}, \quad t\ge 0,  
\end{equation*}
with $\l\in \R$ and $\si>0$, provided that finite time blow-up does
not occur in the case $b=0$, that is, either $\si<2$ or $\l\ge 0$. In
this paper, we extend this study to several directions:
\begin{itemize}
\item The two-dimensional case is considered too.
\item The space variable may belong to the whole space $\R^d$,
  provided that a confining potential is present.
\item When finite time blow-up is present without damping, we
  introduce a superlinear damping in order to prevent blow-up.
\end{itemize}
This last point is related to some conclusion from \cite{AnCaSp-p}: a
nonlinear damping term whose power is larger than that of a focusing
nonlinearity always prevents finite time blow-up.  
\smallbreak

We consider the equation
\begin{equation}
  \label{eq:main}
i\d_t u +\frac{1}{2}\Delta u = V(x)u + \l |u|^{2\si_1}u -ia
|u|^{2\si_2}u-ib\frac{u}{|u|^\alpha}, \quad t\ge 0, \ x\in M,  
\end{equation}
where $V\in C^\infty(M;\R)$ is a smooth, real-valued potential, with initial datum
\begin{equation}
  \label{eq:ci}
  u_{\mid t=0}=u_0.
\end{equation}
Throughout all this paper, we suppose that the following assumption is
satisfied. 
\begin{assumption}\label{hyp:main}
The parameters of the equation are chosen as follows: $\l\in \R$,
$a\ge 0$, $b,\si_1, \si_2>0$, and $\alpha\in [0,1]$. 
  We suppose that $M$ is $d$-dimensional, with $d=1$ or $2$.
\begin{itemize}
\item Either $M$ is a $d$-dimensional compact manifold without boundary, 
\item or $M=\R^d$, and $V$ is harmonic,
  \begin{equation*}
    V(x)=\sum_{j=1}^d \om_j^2 x_j^2, \quad \om_j>0. 
  \end{equation*}
\end{itemize}
If $M=\R^2$, we restrict the range for $\alpha$: $\alpha\in [0,\frac{1}{2}]$. 
\end{assumption}
\begin{remark}
  In the case where $M=\R^d$, we could consider more general
  potentials. Our proofs remain valid provided that $V$ is at most
  quadratic in the sense of \cite{Fujiwara}, that is:
  \begin{equation*}
    V\in C^\infty(\R^d;\R),\quad \text{with}\quad\d^\gamma V\in L^\infty(\R^d),\quad
    \forall \gamma\in \N^d,\ |\gamma|\ge 2.
  \end{equation*}
This assumption is sufficient to construct a global weak solution to
\eqref{eq:main}. We also need the potential energy to control lower
Lebesgue norms (see Lemma~\ref{lem:GNdual}), a requirement which is
satisfied provided that there exist $C,\eps>0$ such that
\begin{equation*}
  V(x)\ge C |x|^{1+\eps},\quad \forall x\in \R^d, \ |x|\ge 1. 
\end{equation*}
Among other properties, such potentials prevents global in time
dispersion (they are confining potentials). It is not clear whether
this assumption is really necessary or if it is a technical
requirement, in order for the conclusions of the present paper to hold. 
\end{remark}
The initial datum satisfies $u_0\in \Sigma$, where
\begin{equation*}
  \Sigma^k = \left\{ f\in H^k(M),\quad \|f\|_{\Sigma^k}^2:=
    \|f\|_{H^k(M)}^2 + \||x|^k f\|_{L^2(M)}^2<\infty\right\},
\end{equation*}
and we denote $\Sigma=\Sigma^1$. Note that if $M$ is
compact, we simply have $\Sigma^k=H^k(M)$, and on $M=\R^d$,
$\Sigma^k = H^k\cap \F(H^k)$, where $\F$ denotes the Fourier transform
(whose normalization is irrelevant in this definition).

\begin{definition}[Weak solution, case $0\le\alpha<1$]
Suppose $0\le\alpha <1$. A (global) weak solution to
  \eqref{eq:main}  is a function $u \in {\mathcal C}(\R_+;L^2(M))\cap
  L^\infty(\R_+; \Sigma)$ 
  solving \eqref{eq:main} in 
  ${\mathcal D}'(\R_+^*\times M)$.
\end{definition}
\begin{definition}[Weak solution, case $\alpha=1$]
  Suppose $\alpha =1$. A (global) weak solution to
  \eqref{eq:main}  is a function $u \in {\mathcal C}(\R_+;L^2(M))\cap
  L^\infty(\R_+; \Sigma)$
solving
  \begin{equation*}
  i\d_t u+\frac{1}{2}\Delta u=   V(x)u + \l |u|^{2\si_1}u -ia
|u|^{2\si_2}u-ib F  
  \end{equation*}
in ${\mathcal D}'(\R_+^*\times M)$, where $F$ is such that
\begin{equation*}
 \|F\|_{L^\infty(\R_+\times M )} \le 1,\quad \text{and}\quad
 F=\frac{u}{|u|} \text{ if } u\neq 0.
\end{equation*}
\end{definition}
\begin{theorem}\label{theo:exist}
  Let $u_0\in \Sigma$. In
  either of the following cases,
  \begin{itemize}
  \item $\si_1<2/d$,
\item or $\l\ge 0$,
\item or $\l<0$, $a>0$ and $\si_2>\si_1$,
  \end{itemize}
the Cauchy problem \eqref{eq:main}-\eqref{eq:ci} has a unique,
global, weak solution. 
\end{theorem}
Multiplying \eqref{eq:main} by $\bar u$,
integrating over $M$ and taking the imaginary part, we obtain formally:
\begin{equation}
  \label{eq:mass}
  \frac{d}{dt}\|u(t)\|_{L^2}^2 +2a \int_M |u(t,x)|^{2\si_2+2}dx +2b
    \int_M |u(t,x)|^{2-\alpha}dx =0 . 
\end{equation}
We will check in the course of the proof of Theorem~\ref{theo:exist}
that the solution satisfies this relation indeed. 
\begin{corollary}\label{cor:one}
  Let $d=1$ and $\alpha>0$ in Assumption~\ref{hyp:main}, and $u_0\in \Sigma$. In
  either of the cases considered in Theorem~\ref{theo:exist}, there
  exists $T>0$ such that the unique weak solution to
  \eqref{eq:main}-\eqref{eq:ci} satisfies
  \begin{equation*}
    \text{for every } t\ge T,\quad \|u(t)\|_{L^2(M)}=0. 
  \end{equation*}
\end{corollary}

\begin{theorem}\label{theo:two}
  Let $d=2$ in Assumption~\ref{hyp:main}, and
  $u_0\in \Sigma$.\\
(1) In
  either of the cases considered in Theorem~\ref{theo:exist}, there
  exists $C>0$ such that the solution to \eqref{eq:main}-\eqref{eq:ci}
  satisfies
  \begin{equation*}
    \|u(t)\|_{L^2(M)}\le \|u_0\|_{L^2(M)} e^{-Ct},\quad t\ge 0.
  \end{equation*}
(2) If in addition $u_0\in \Sigma^2$, then $u\in
L^\infty(\R_+;\Sigma^2)$. If $1/2\le\si_1\le 3/2$, then for
any $R>0$, there exists $\eta_R>0$ such that if $\|u_0\|_{\Sigma^2}\le
R$ and 
$\|u_0\|_{L^2}\le \eta_R$, then there exists $T>0$ such that for every
$t\ge T$,  $\|u(t)\|_{L^2(M)}=0$. 
 \end{theorem} 
Note that the above smallness assumption is automatically fulfilled as
soon as $\|u_0\|_{\Sigma^2}$ is sufficiently small. 

The proof of the second part of this theorem relies on
Br\'ezis-Gallou\"et inequality introduced in \cite{BrGa80} (and
recently revisited in \cite{OzVi-p}), which require higher energy
estimates.

\section{Existence result and a priori estimates}
\label{sec:exist}

\subsection{Preliminary technical results}
\label{sec:prelim}

We recall the standard Gagliardo-Nirenberg inequalities (see e.g. \cite{CazCourant}):
\begin{lemma}\label{lem:GN}
  Let $M$ be as in Assumption~\ref{hyp:main}.  If $d=1$, let $p\in [2,\infty]$, and if
  $d=2$, let $p\in [2,\infty)$. There exists
  $C=C(p,d)$ such that for all $f\in H^1(M)$,
  \begin{equation*}
    \|f\|_{L^p(M)} \le C
    \|f\|_{L^2(M)}^{1-\delta(p)}\|f\|_{H^1(M)}^{\delta(p)},\quad \text{where}\quad
    \delta(p)= d\(\frac{1}{2}-\frac{1}{p}\).
  \end{equation*}
If $M=\R^d$, then the inhomogeneous Sobolev norm
$\|\cdot\|_{H^1(\R^d)} $ can be replaced by the homogeneous norm
$\|\cdot\|_{\dot H^1(\R^d)} $.
\end{lemma}
We recall the standard compactness result (see e.g. \cite{KavianWeissler}):
\begin{lemma}\label{lem:compact}
  Let $M=\R^d$, $d=1$ or $2$. If $d=1$, let $p\in [2,\infty]$, and if
  $d=2$, let $p\in [2,\infty)$. The embedding $\Sigma \hookrightarrow
  L^p(\R^d)$ is compact. 
\end{lemma}

If $M$ is a compact manifold without boundary, H\"older inequality
readily yields, for $1\le p<q\le \infty$,
\begin{equation*}
  \|f\|_{L^p(M)}\le |M|^{1/p-1/q} \|f\|_{L^q(M)},\quad \forall f\in L^q(M). 
\end{equation*}
On the whole space $\R^d$, an analogous inequality is
provided by the control of momenta, which can be
viewed as dual to the 
Gagliardo-Nirenberg inequalities (see e.g. \cite{CaMi04}): 
\begin{lemma}\label{lem:GNdual}
  Let $M=\R^d$, $d=1$ or $2$. If $d=1$, let $p\in [2,\infty]$, and if
  $d=2$, let $p\in [2,\infty)$. There exists $C=C(p,d)$ such that for
  all $f\in \F(H^1(\R^d))$, 
  \begin{equation*}
    \|f\|_{L^{p'}(\R^d)}\le C
    \|f\|_{L^2(\R^d)}^{1-\delta(p)}\|x f\|_{L^2(\R^d)}^{\delta(p)},\quad
    \text{where}\quad 
    \delta(p)= d\(\frac{1}{2}-\frac{1}{p}\).
  \end{equation*}
\end{lemma}

\subsection{Approximate solution}
\label{sec:approximate}
 Following the same strategy as in \cite{CaGa11}, we modify
 \eqref{eq:main} by regularizing the sublinear nonlinearity:
\begin{equation}
  \label{eq:approx}
i\d_t u^\delta +\frac{1}{2}\Delta u^\delta = V(x)u^\delta + \l |u^\delta|^{2\si_1}u^\delta -ia
|u^\delta|^{2\si_2}u^\delta-ib\frac{u^\delta}{\(|u^\delta|^2+\delta\)^{\alpha/2}}.
\end{equation}
We keep the same initial datum \eqref{eq:ci}. Since the external
potential $V$ is at most quadratic, local in time Strichartz
inequalities are available for the Hamiltonian $-\frac{1}{2}\Delta
+V$. With $d\le 2$, all the nonlinearities are energy-subcritical,
and we infer (see e.g. \cite{CazCourant}):
\begin{lemma}\label{lem:cauchyloc}
  Let  $\delta>0$, and  $u_0\in
  \Sigma$. There exist $T>0$ and a unique solution 
  \begin{equation*}
    u^\delta  \in C([0,T];\Sigma)\cap
    L^{\frac{4\si_1+4}{d\si_1}}([0,T];L^{2\si_1+2}(M)) \cap
    L^{\frac{4\si_2+4}{d\si_2}}([0,T];L^{2\si_2+2}(M)) 
  \end{equation*}
to the Cauchy problem \eqref{eq:approx}-\eqref{eq:ci}. In addition,
for all $t\in [0,T]$, it satisfies
\begin{equation}
  \label{eq:L2}
  \|u^\delta(t)\|_{L^2(M)}^2 + 2b\int_0^t\int_M
  \frac{|u^\delta(\tau,x)|^2}{\(|u^\delta(\tau,x)|^2+\delta\)^{\alpha/2}}dxd\tau
  \le  \|u_0\|_{L^2(M)}^2. 
\end{equation}
\end{lemma}
To prove that the solution to \eqref{eq:approx} is actually global in
the future (the equation is irreversible), denote by
\begin{equation}\label{eq:E0}
  E_0^\delta(t) = \|\nabla u^\delta(t)\|_{L^2}^2+ 2\int_M
  V(x)|u^\delta(t,x)|^2dx+\frac{2\lambda}{\si_1+1} 
  \|u^\delta(t)\|_{L^{2\si_1+2}}^{2\si_1+2},
\end{equation}
and, following the approach introduced in \cite{AnSp10}, for $k>0$, set
\begin{equation}\label{eq:Ek}
  E_k^\delta(t):=E_0^\delta(t) + k \|u^\delta(t)\|_{L^{2\si_2+2}}^{2\si_2+2}.
\end{equation}
The energy $E_0^\delta$ involves the Hamiltonian part of
\eqref{eq:approx}, and $E_k^\delta$ consists of the artificial
introduction of the extra nonlinearity $|u|^{2\si_2}u$, as if it were
Hamiltonian instead of a damping term. 
\begin{proposition}\label{prop:apriori}
 (1)  Assume that $\si_1<2/d$ or $\l \ge 0$. 
 There exists a $C=C(\|u_0\|_{L^2})\ge 0$ independent of
$\delta\in (0,1]$ such that
\begin{equation*}
  E_0^\delta(t)\le E_0^\delta(0)+C(\|u_0\|_{L^2})\quad \forall t\in [0,T],
\end{equation*}
where $T>0$ is a local existence time in $\Sigma$. \\
(2) If $\l<0$, assume that $a>0$, $0<k<\frac{2a}{\si_2(\si_2+1)}$, and $\si_2>\si_1$.
There exists a $C=C(\|u_0\|_{L^2})\ge 0$ independent of
$\delta\in (0,1]$  such that
\begin{equation*}
  E_k^\delta(t)\le E_k^\delta(0)+C(\|u_0\|_{L^2})\quad \forall t\in [0,T],
\end{equation*}
where $T>0$ is a local existence time in $\Sigma$. 
 \end{proposition}
\begin{proof}
Denote by 
\begin{equation*}
  f_\delta(v) = \frac{v}{\(|v|^2+\delta\)^{\alpha/2}}.
\end{equation*}
Since
\begin{equation*}
  \Delta u^\delta = -2i\d_t u^\delta +2Vu^\delta +2\l
  |u^\delta|^{2\si_1} u^\delta -2ia |u^\delta|^{2\si_2}u^\delta 
  -2ibf_\delta(u^\delta), 
\end{equation*}
we compute, with $(f\mid g)=\int_M f\bar g$, 
\begin{align*}
  \frac{d}{dt}\|\nabla u^\delta(t)\|_{L^2}^2& = -2\RE\(\d_t
  u^\delta\mid \Delta u^\delta\)\\
&= -2\RE\(\d_t u^\delta\mid -2i\d_t u^\delta +2Vu^\delta +2\l
|u^\delta|^{2\si_1} u^\delta -2ia |u^\delta|^{2\si_2}u^\delta 
  -2ibf_\delta(u^\delta)\)\\
&= -\frac{d}{dt}\(2\int_M V|u^\delta|^2+\frac{2\lambda}{\si_1+1}
\|u^\delta\|_{L^{2\si_1+2}}^{2\si_1+2}\) +4 a\IM\(\d_t 
u^\delta\mid |u^\delta|^{2\si_2}u^\delta\) \\
&\quad+4b\IM\(\d_t u^\delta\mid f_\delta(u^\delta)\). 
\end{align*}
Since 
\begin{equation*}
  \d_t u^\delta = \frac{i}{2}\Delta u^\delta -iV u^\delta -i \l
  |u^\delta|^{2\si_1} u^\delta -a  |u^\delta|^{2\si_2}u^\delta -b
  f_\delta(u^\delta), 
\end{equation*}
we have
\begin{align*}
  \IM\(\d_t u^\delta\mid |u^\delta|^{2\si_2}u^\delta\) &=
  \frac{1}{2}\RE \(\Delta u^\delta \mid 
  |u^\delta|^{2\si_2}u^\delta\) -\int_M V|u^\delta|^{2\si_2+2} -\l
  \int_M|u^\delta|^{2\si_1+2\si_2+2}\\ 
& = -\frac{1}{2}\int_M |u^\delta|^{2\si_2}|\nabla u^\delta |^2 -\si_2 \int_M
|u^\delta|^{2\si_2}\left\lvert \nabla \lvert u^\delta\rvert\right\rvert^2 -\int_M
V|u^\delta|^{2\si_2+2}\\
&\quad  -\l \int_M|u^\delta|^{2\si_1+2\si_2+2}, 
\end{align*}
where for the last equality, we have used the identity
\begin{equation*}
  \Delta |u^\delta|^2 = 2\RE\(\bar u^\delta\Delta u^\delta\)+2|\nabla u^\delta|^2. 
\end{equation*}
On the other hand, we have
\begin{equation*}
  \IM\(\d_t u^\delta\mid f_\delta(u^\delta)\)=\frac{1}{2}\RE \(\Delta u^\delta
  \mid f_\delta(u^\delta)\) - \int_M
  V\frac{|u^\delta|^{2}}{\(|u^\delta|^2+\delta\)^{\alpha/2}} -\l
  \int_M\frac{|u^\delta|^{2\si_1+2}}{\(|u^\delta|^2+\delta\)^{\alpha/2}}. 
\end{equation*}
We have
\begin{align*}
  \RE \(\Delta u^\delta
  \mid f_\delta(u^\delta)\) & = -\RE \(\nabla u^\delta\mid\nabla
  f_\delta(u^\delta)\)\\
&= -\int_M\frac{|\nabla
  u^\delta|^2}{\(|u^\delta|^2+\delta\)^{\alpha/2}}+\alpha \RE \int_M
\bar u^\delta \nabla u^\delta\cdot \frac{\RE(\bar u^\delta\nabla
 u^\delta)}{\(|u^\delta|^2+\delta\)^{\alpha/2+1}} \\
&= 
-\int_M\(|u^\delta|^2+\delta\)\frac{|\nabla
  u|^2}{\(|u^\delta|^2+\delta\)^{\alpha/2+1}}+\alpha\int_M
 \frac{\left|\RE(\bar u^\delta\nabla
 u^\delta)\right|^2}{\(|u^\delta|^2+\delta\)^{\alpha/2+1}}\\
& =
-\delta\int_M\frac{|\nabla
  u|^2}{\(|u^\delta|^2+\delta\)^{\alpha/2+1}}-\int_M
 \frac{\left|\IM(\bar u^\delta\nabla
 u^\delta)\right|^2}{\(|u^\delta|^2+\delta\)^{\alpha/2+1}}\\
&\quad -(1-\alpha)\int_M
 \frac{\left|\RE(\bar u^\delta\nabla
 u^\delta)\right|^2}{\(|u^\delta|^2+\delta\)^{\alpha/2+1}}.
\end{align*}
From the above computations, we have:
\begin{align*}
  \frac{d}{dt}E_0^\delta &= -2a \int_M |u^\delta|^{2\si_2}|\nabla u^\delta |^2 -4a\si_2 \int_M
|u^\delta|^{2\si_2}\left\lvert \nabla \lvert u^\delta\rvert\right\rvert^2 -4a \int_M
V|u^\delta|^{2\si_2+2}\\
&\quad -4a\l \int_M|u^\delta|^{2\si_1+2\si_2+2}  -4b \int_M
  V\frac{|u^\delta|^{2}}{\(|u^\delta|^2+\delta\)^{\alpha/2}} -4b\l
  \int_M\frac{|u^\delta|^{2\si_1+2}}{\(|u^\delta|^2+\delta\)^{\alpha/2}}\\
&\quad -\delta\int_M\frac{|\nabla
  u|^2}{\(|u^\delta|^2+\delta\)^{\alpha/2+1}}-\int_M
 \frac{\left|\IM(\bar u^\delta\nabla
 u^\delta)\right|^2}{\(|u^\delta|^2+\delta\)^{\alpha/2+1}}\\
&\quad-2b\delta\int_M\frac{|\nabla
  u|^2}{\(|u^\delta|^2+\delta\)^{\alpha/2+1}}-2b\int_M
 \frac{\left|\IM(\bar u^\delta\nabla
 u^\delta)\right|^2}{\(|u^\delta|^2+\delta\)^{\alpha/2+1}}\\
&\quad -2b(1-\alpha)\int_M
 \frac{\left|\RE(\bar u^\delta\nabla
 u^\delta)\right|^2}{\(|u^\delta|^2+\delta\)^{\alpha/2+1}}
 .
\end{align*}
If $\l\ge 0$ (defocusing case), $E_0^\delta$, defined in
\eqref{eq:E0}, is non-increasing. If $\si_1<2/d$, we conclude as in
the standard case presented for instance in \cite{CazCourant}.

To treat the 
focusing case $\l<0$, with $\si_1\ge 2/d$ (finite time blow-up is
possible in the case $a=b=0$), we follow the strategy adopted in \cite{AnSp10} and
generalized in \cite{AnCaSp-p}, relying on $E_k^\delta$, defined in \eqref{eq:Ek}. 
The following computation is valid for any $p>2$:
\begin{align*}
  \frac{d}{dt}\|u^\delta(t)\|_{L^p}^p& = p\RE\(\d_t u^\delta\mid |u^\delta|^{p-2}u^\delta\)\\
& = -\frac{p}{2}\IM \(\Delta u^\delta\mid|u^\delta|^{p-2}u^\delta\) -ap \int_M
|u^\delta|^{2\si_2+p}-bp\int_M \frac{|u^\delta|^{p}}{\(|u^\delta|^2+\delta\)^{\alpha/2}}\\
&= \frac{p}{2}\int_M \nabla|u^\delta|^{p-2}\cdot \IM\(\bar
u^\delta\nabla u^\delta\)-ap \int_M 
|u^\delta|^{2\si_2+p}-bp\int_M \frac{|u^\delta|^{p}}{\(|u^\delta|^2+\delta\)^{\alpha/2}}.
\end{align*}
As in \cite{AnCaSp-p},  we use the polar factorisation introduced in
\cite{Ka72,Ag82} 
(see also \cite{AnMa09,CaDaSa12}), to show that 
\begin{equation*}
\int_M \nabla|u^\delta|^{p-2}\cdot \IM\(\bar u^\delta\nabla u^\delta\)=
(p-2)\int_M|u^\delta|^{p-2}\RE(\bar\phi\nabla u^\delta)\cdot\IM(\bar\phi\nabla
u^\delta), 
\end{equation*}
where $\phi$ is the polar factor related to $u^\delta$, 
\begin{equation*}
  \phi(t,x):=
\left\{
  \begin{aligned}
   & |u^\delta(t,x)|^{-1}u^\delta(t,x) & \quad\text{ if }u^\delta(t,x)\not =0,\\
&0& \quad\text{ if }u^\delta(t,x) =0.
  \end{aligned}
\right.
\end{equation*}
In view of the identity
\begin{equation*}
 2 \RE(\bar\phi\nabla u^\delta)\cdot\IM(\bar\phi\nabla
u^\delta) = -\left| \RE(\bar\phi\nabla u^\delta) - \IM(\bar\phi\nabla
u^\delta)\right|^2 +|\nabla u^\delta|^2,
\end{equation*}
we obtain:
\begin{align*}
   \frac{d}{dt}\|u^\delta(t)\|_{L^p}^p&= -\frac{p(p-2)}{4}\int_M |u^\delta|^{p-2}
   \left| \RE(\bar\phi\nabla u^\delta) - \IM(\bar\phi\nabla u^\delta)\right|^2\\
&\quad+\frac{p(p-2)}{4} \int_M |u^\delta|^{p-2}|\nabla u^\delta|^2 -ap \int_M
|u^\delta|^{2\si_2+p}-bp\int_M \frac{|u^\delta|^{p}}{\(|u^\delta|^2+\delta\)^{\alpha/2}}.
\end{align*}
We finally have:
\begin{align*}
  \frac{d}{dt}E_k^\delta&\le -2a \int_M |u^\delta|^{2\si_2}|\nabla u^\delta |^2 -4a\l
  \int_M|u^\delta|^{2\si_1+2\si_2+2} -4b\l\int_M\frac{
    |u^\delta|^{2\si_1+2}}{\(|u^\delta|^2+\delta\)^{\alpha/2}} \\ 
&\quad +k\si_2(\si_2+1) \int_M |u^\delta|^{2\si_2}|\nabla u^\delta |^2-a k(2\si_2+2)\int_M
|u^\delta|^{4\si_2+2}\\
&\quad  -bk(2\si_2+2)\int_M\frac{ |u^\delta|^{2\si_2+2}}{\(|u^\delta|^2+\delta\)^{\alpha/2}}. 
\end{align*}
If $0<k\si_2(\si_2+1)<2a$, and since $\l<0$, we come up with:
\begin{align*}
  \frac{d}{dt}E_k^\delta&\le 4a|\l|
  \int_M|u^\delta|^{2\si_1+2\si_2+2} +4b|\l|\int_M\frac{
    |u^\delta|^{2\si_1+2}}{\(|u^\delta|^2+\delta\)^{\alpha/2}}  \\
&\quad -a k(2\si_2+2)\int_M
|u^\delta|^{4\si_2+2}  -bk(2\si_2+2)\int_M \frac{
  |u^\delta|^{2\si_2+2}}{\(|u^\delta|^2+\delta\)^{\alpha/2}}.  
\end{align*}
If $\si_2>\si_1$ (the superlinear damping is ``stronger'' than the focusing
term), then the negative terms on the right hand side control the
positive terms (since $\|u^\delta\|_{L^2}$ is non-increasing), hence
the result.
\end{proof}

\subsection{Convergence of the approximation}
\label{sec:convergence}

We now follow the strategy introduced in \cite{GV85c}, and resumed in
\cite{CaGa11}. 

A straightforward
consequence from \eqref{eq:L2} and Proposition~\ref{prop:apriori} is that for 
$u_0\in \Sigma$ fixed, the sequence $(u^\delta)_{0<\delta\le 1}$ is
uniformly bounded in $L^\infty(\R_+,\Sigma)\cap
L^{2-\alpha}(\R_+\times M)$. We deduce the existence of
$u\in L^\infty(\R_+,\Sigma)$ and of a
subsequence $u^{\delta_n}$ such that
\begin{equation}\label{convf*}
u^{\delta_n}\rightharpoonup u,\quad  \text{in }w*\ L^\infty(\R_+,\Sigma),
\end{equation}
with, in view of \eqref{eq:L2} and Proposition~\ref{prop:apriori}, 
\begin{equation*}
  \|u\|_{L^\infty(\R_+,H^1( M))}\le \|u_0\|_{H^1(M)} +C(\|u_0\|_{L^2(M)}).
\end{equation*}
Moreover,
$\frac{u^\delta}{(|u^\delta|^2+\delta)^{\alpha/2}}$ is uniformly bounded
in $L^\infty(\R_+,L^{\frac{2}{1-\alpha}}( M))$ (with $2/(1-\alpha)=\infty$
if $\alpha=1$), such that up to the
extraction of an 
other subsequence, there is $F\in
L^\infty(\R_+,L^{\frac{2}{1-\alpha}}( M))$ 
such that
\begin{equation}\label{convf*F}
\frac{u^{\delta_n}}{(|u^{\delta_n}|^2+\delta_n)^{\alpha/2}}
\rightharpoonup F,\quad \text{in } w*\ L^\infty(\R_+,L^{\frac{2}{1-\alpha}}( M)).
\end{equation}
Moreover, $\|F\|_{L^\infty(\R_+,L^{\frac{2}{1-\alpha}}( M))}\le
\|u_0\|_{L^2(M)}^{1-\alpha}$. In view of Lemma~\ref{lem:compact}
(whose analogue  is obvious in the case where $M$ is compact), 
\begin{equation*}
  |u^{\delta_n}|^{2\si_j}u^{\delta_n} \Tend n \infty
  |u|^{2\si_j}u\quad\text{in }L^1_{\rm loc}(\R_+\times M),\quad j=1,2. 
\end{equation*}
Let $\theta\in \mathcal{C}_c^\infty(\R_+^*\times M)$. Then
\begin{align*}
&\<-ib\frac{u^{\delta_n}}{(|u^{\delta_n}|^2+\delta_n)^{\alpha/2}},
  \theta\>\\
&=\<i\d
  u^{\delta_n}+\frac{1}{2}\Delta
u^{\delta_n}-Vu^{\delta_n} -\l |u^{\delta_n}|^{2\si_1}u^{\delta_n}
+ia  |u^{\delta_n}|^{2\si_2}u^{\delta_n},\theta\>\\
&=\<u^{\delta_n},-i\frac{\partial
  \theta}{\partial t}+\frac{1}{2}\Delta \theta\>+ \<-Vu^{\delta_n} -\l
|u^{\delta_n}|^{2\si_1}u^{\delta_n} 
+ia  |u^{\delta_n}|^{2\si_2}u^{\delta_n},\theta\>\\
&\Tend n \infty \<u,-i\frac{\partial
  \theta}{\partial t}+\frac{1}{2}\Delta \theta\>
+ \<-Vu -\l
|u|^{2\si_1}u
+ia  |u|^{2\si_2}u,\theta\>\\
&=\<i\frac{\partial
  u}{\partial t}+\frac{1}{2}\Delta u-Vu -\l |u|^{2\si_1}u
+ia  |u|^{2\si_2}u,\theta\>,
\end{align*}
where $\<\cdot,\cdot\>$ stands for the distribution bracket on
$\R_+^*\times M$. Thus, we deduce 
$$i\d_t u+\frac{1}{2}\Delta u= V(x)u + \l |u|^{2\si_1}u -ia
|u|^{2\si_2}u-ib F,\quad \text{ in }
\mathcal{D}'(\R_+^*\times M).$$
We next show that $F= u/|u|^\alpha$ where the right hand side is well
defined, that is if $\alpha<1$, or $\alpha=1$ and $u\neq 0$. We first
suppose that $u_0\in 
H^s( M)$ with $s$ large. Let us fix $t'\in \R_+$ and $\delta>0$. Thanks to
\eqref{eq:L2}, we infer, for any $t\in \R_+$,
\begin{equation*}
\frac{d}{dt}\|u^\delta(t)-u^\delta(t')\|_{L^2}^2  \le 
\frac{d}{dt}\big(-2\RE
\left(u^\delta(t)\mid u^\delta(t')\right)\big),
\end{equation*}
where $(\cdot\mid \cdot)$ denotes the scalar product in $L^2(M)$. In view
of \eqref{eq:approx}, the right hand side is equal to
\begin{equation*}
  -2\RE\left(\frac{i}{2}\Delta u^\delta(t)-iVu^\delta -i\l
  |u^\delta|^{2\si_1}u^\delta -a|u^\delta|^{2\si_2}u^\delta-\frac{b
    u^\delta(t)}{(|u^\delta(t)|^2+\delta)^{\alpha/2}}\Big|u^\delta(t')\right).
\end{equation*}
 By
integration, we deduce
\begin{equation}\label{cont}
  \begin{aligned}
   \|u^\delta(t)-u^\delta(t')\|_{L^2( M)}^2& \le 2|t-t'|\Big( \frac{1 }{2}\|\Delta
u^\delta\|_{L^\infty(\R_+;H^{-1})}
\|u^\delta\|_{L^\infty(\R_+;H^{1})}\\
&\quad+\|Vu\|_{L^\infty(\R_+;L^2)}^2 + |\l|
\|u^\delta\|_{L^\infty(\R_+;L^{2\si_1+2})}^{2\si_1+2} \\
&\quad+a
\|u^\delta\|_{L^\infty(\R_+;L^{2\si_2+2})}^{2\si_2+2}+ 
b\|u^\delta\|^{2-\alpha}_{L^\infty(\R_+,L^{2-\alpha}( M))}\Big).
  \end{aligned}
\end{equation}
From the continuity of the flow map $\Sigma\ni u_0\mapsto u^\delta\in
\mathcal{C}(\R_+,\Sigma)$ in Lemma~\ref{lem:cauchyloc}, we deduce that
\eqref{cont} also holds if we only have $u_0\in \Sigma$. Next,
since $(u^\delta)_{0<\delta\le 1}$ is uniformly bounded in
$L^\infty(\R_+,\Sigma)$ and either  $ M$ is compact or we may invoke
Lemma~\ref{lem:GNdual} (recall that on $\R^2$, we assume $\alpha\le
1/2$ in Assumption~\ref{hyp:main}),  
\eqref{cont} gives the existence of a positive constant $C$ such that
for every $t,t'\in \R_+$,
$$\|u^\delta(t)-u^\delta(t')\|_{L^2( M)}\le C|t-t'|^{1/2}.$$
In particular, for any $T>0$, $(u^\delta)_{0<\delta\le 1}$ is a
bounded sequence in $\mathcal{C}([0,T],L^2( M))$ which is uniformly
equicontinuous from $[0,T]$ to $L^2( M)$. Moreover, the compactness of
the embedding $\Sigma\subset L^2(M)$ ensures that for every $t\in
[0,T]$, the set $\{u^\delta(t)|\delta\in (0,1]\}$ is relatively compact in $L^2(M)$. As a result, Arzel\`a--Ascoli
Theorem implies that $(u^{\delta_n})_n$ is relatively compact in
$\mathcal{C}([0,T],L^2( M))$. On the other hand, we already know
from \eqref{convf*} that 
$$u^{\delta_n}\rightharpoonup u\quad \text{in } w*\ L^\infty(\R_+,L^2( M)).$$
Therefore, we infer that $u$ is the unique accumulation point of the sequence
$(u^{\delta_n})_n$ in $\mathcal{C}([0,T], L^2( M))$. Thus
$$u^{\delta_n}\to u\quad \text{in }\mathcal{C}([0,T],L^2( M)),$$
which implies in particular $u\in \mathcal{C}([0,T],L^2( M))$
as well as $u(0)=u^{\delta_n}(0)=u_0$. This is true for any $T>0$,
therefore 
$$u\in \mathcal{C}(\R_+,L^2( M)).$$
Finally, up to the extraction of an other subsequence, 
$u^{\delta_n}(t,x)\to u(t,x)$ for almost every $(t,x)\in
\R_+\times M$. Therefore, for almost every $(t,x)\in
\R_+\times M$ such that $u(t,x)\neq 0$, we have
$$\frac{u^{\delta_n}}{(|u^{\delta_n}|^2+\delta_n)^{\alpha/2}}(t,x)\to
\frac{u}{|u|^\alpha}(t,x).$$
By comparison with \eqref{convf*F}, we deduce that up to a change of
$F$ on a set with zero measure,
$$F(t,x)=\frac{u}{|u|^\alpha}(t,x)\quad\text{(only if }u(t,x)\neq
0\text{ in the case }\alpha=1\text{)},$$ 
which completes the proof of the existence part 
of Theorem~\ref{theo:exist}. 

\subsection{Uniqueness}
\label{sec:uniqueness}

If $u$ and $v$ are two solutions to \eqref{eq:main}, then by
subtracting the two equations, multiplying by $\overline{u-v}$,
integrating over $M$ and taking the imaginary part, we obtain:
\begin{equation}\label{eq:unique}
\begin{aligned}
  \frac{d}{dt}&\|u-v\|_{L^2}^2 +2a \RE\int_M
  \(|u|^{2\si_2}u-|v|^{2\si_2}v\)\overline{u-v}\\
&+ 2 b \RE\int_M
  \(\frac{u}{|u|^{\alpha}}-\frac{v}{|v|^{\alpha}}\)\overline{u-v}= 2\l
    \IM \int_M \(|u|^{2\si_1}u-|v|^{2\si_1}v\)\overline{u-v}.
\end{aligned}
\end{equation}
Extending Lemma~3.1 from \cite{CaGa11}, we have
\begin{lemma}\label{lem:young}
  Let $\si\ge -1$. For all $z_1,z_2\in \C$,
  \begin{equation*}
    \RE \( \(|z_1|^\sigma z_1-|z_2|^{\sigma}z_2\)\(\overline{z_1-z_2}\)\)\ge 0.
  \end{equation*}
\end{lemma}
\begin{proof}
  Using polar coordinates, write $z_j=\rho_j e^{i\theta_j}$,
  $\rho_j\ge 0$, $\theta_j\in \R$. The quantity involved in the
  statement is
  \begin{equation*}
    \rho_1^{\si+2}
     +\rho_2^{\si+2} -\rho_1^{\si+1}\rho_2 \cos (\theta_1-\theta_2) -
     \rho_2^{\si+1}\rho_1 \cos (\theta_1-\theta_2). 
  \end{equation*}
Since the cosine function  is bounded by one, the above quantity is
bounded from below by
\begin{equation*}
   \rho_1^{\si+2}
     +\rho_2^{\si+2} -\rho_1^{\si+1}\rho_2-
     \rho_2^{\si+1}\rho_1 =
     \(\rho_1^{\si+1}-\rho_2^{\si+1}\)(\rho_1-\rho_2). 
\end{equation*}
If $\si=-1$, the above quantity is identically zero. If $\si>-1$, then
we conclude by observing that both factors on the right hand side
always have the same sign. 
\end{proof}
If $d=1$, \eqref{eq:unique} and the above lemma yield
\begin{align*}
  \frac{d}{dt}\|u(t)-v(t)\|_{L^2}^2 &\le 2|\l|
     \int_M
     \left|\(|u|^{2\si_1}u-|v|^{2\si_1}v\)\overline{u-v}\right|\\
&\le
     C\(\|u\|_{L^\infty H^1}^{2\si_1} + \|v\|_{L^\infty H^1}^{2\si_1} \)\|u(t)-v(t)\|_{L^2}^2,
\end{align*}
and Gronwall lemma shows that there is at most one (global) weak
solution to \eqref{eq:main}. 
\smallbreak

When $d=2$, in order to overcome the absence of control in
$L^\infty(M)$, we invoke the argument introduced by Yudovitch \cite{Yu63}, and
resumed in the context of nonlinear Schr\"odinger equations in
\cite{OgOz91,Oz95},  and 
by Burq,
G\'erard and Tzvetkov \cite{BGT} in the case of three-dimensional
domains. Since their argument readily works 
in the present context, we simply recall it. 

Denote by $\epsilon(t) = \|u(t)-v(t)\|_{L^2(M)}^2$. For $p$ finite and
large, \eqref{eq:unique}, Lemma~\ref{lem:young} and H\"older
inequality yield 
\begin{align*}
  \dot \epsilon(t)& \le C \int_M \(
  |u(t,x)|^{2\si_1}+|v(t,x)|^{2\si_1}\)|u(t,x)-v(t,x)|^2dx\\
&\le C \(\|u(t)\|_{L^{2p\si_1}}^{2\si_1} + \|v(t)\|_{L^{2p\si_1}}^{2\si_1} \)\|u(t)-v(t)\|_{L^{2p'}}^2,
\end{align*}
where the constant $C$ does not depend on $p$. By interpolation,
\begin{equation*}
  \|u(t)-v(t)\|_{L^{2p'}} \le \|u(t)-v(t)\|_{L^2}^{1-3/2p}\|u(t)-v(t)\|_{L^6}^{3/2p},
\end{equation*}
hence, in view of the boundedness of the $L^\infty_tH^1_x$ norm of $u$ and $v$, and
of Sobolev embedding $H^1(M)\hookrightarrow L^6(M)$,
\begin{equation*}
  \dot\epsilon(t) \le C\(\|u(t)\|_{L^{2p\si_1}}^{2\si_1} +
  \|v(t)\|_{L^{2p\si_1}}^{2\si_1} \) \epsilon(t)^{1-3/2p}. 
\end{equation*}
Gagliardo--Nirenberg inequality implies
\begin{equation*}
  \|u(t)\|_{L^{2p\si_1}}^{2\si_1} +
  \|v(t)\|_{L^{2p\si_1}}^{2\si_1} \le C \([p]!\)^{1/p} \(
  \|u(t)\|_{H^1}^{2\si_1} + \|v(t)\|_{H^1}^{2\si_1} \),
\end{equation*}
with another constant $C$, still independent of $p$ (see
e.g. \cite{Taylor3}). Therefore, using Stirling formula for $p$ large,
\begin{equation*}
   \dot\epsilon(t) \le C  p \epsilon(t)^{1-3/2p}. 
\end{equation*}
By integration in time, under the assumption $\epsilon(0)=0$, we come up
with
\begin{equation*}
  \epsilon(t)^{3/2p}\le Ct,
\end{equation*}
for some constant $C$ independent of $p$. Choosing $t$ sufficiently
small and letting $p\to \infty$, we see that $\epsilon=0$ on some interval
$[0,t_0]$ for some universal constant $t_0$, hence $\epsilon\equiv 0$ by
induction. 
\smallbreak

Therefore, there is at most one (global) weak solution to
\eqref{eq:main}--\eqref{eq:ci}. In addition, by considering $v=0$ in
\eqref{eq:unique}, we see that this solution satisfies
\eqref{eq:mass}.

\section{Finite time extinction in 1D and exponential decay in 2D}
\label{sec:one}

The following lemma follows from inequalities on $\R^d$, adapted from the
Nash inequality \cite{Na58} (see \cite{CaGa11}):
\begin{lemma}\label{lem:Nash}
  Let $M$ be as in Assumption~\ref{hyp:main}. Let $\alpha\in
  ]0,1]$. There exists $C>0$ such that 
 \begin{align}
    \|f\|_{L^2(M)}^{\alpha d + 4-2\alpha}&\le C
    \(\|f\|^{2-\alpha}_{L^{2-\alpha}(M)}\)^{2}\|f\|_{H^1(M)}^{\alpha
      d},\quad \forall f\in 
    H^1(M).\label{eq:nash1}\\ 
\|f\|_{L^2(M)}^{\alpha d + 8-4\alpha }&\le C
    \(\|f\|^{2-\alpha}_{L^{2-\alpha}(M)}\)^{4}\|f\|_{H^2(M)}^{\alpha
      d},\quad \forall f\in
    H^2(M).\label{eq:nash2}
  \end{align} 
If $M=\R^d$, then the inhomogeneous Sobolev norm
$\|\cdot\|_{H^s(\R^d)} $ can be replaced by the homogeneous norm
$\|\cdot\|_{\dot H^s(\R^d)} $.
\end{lemma}

\subsection{Proof of Corollary~\ref{cor:one}}
\label{sec:proof1D}

Suppose that $d=1$ in Theorem~\ref{theo:exist}. 
In view of \eqref{eq:mass}, we have
\begin{equation*}
   \frac{d}{dt}\|u(t)\|_{L^2}^2 +2b
    \int_M |u(t,x)|^{2-\alpha}dx \le 0 . 
\end{equation*}
Theorem~\ref{theo:exist} and Lemma~\ref{lem:Nash} yield
\begin{equation*}
  \frac{d}{dt}\|u(t)\|_{L^2}^2 +Cb \|u(t)\|_{L^2}^{2-\alpha/2}\le 0, 
\end{equation*}
where $C$ is proportional to
$\|u\|_{L^\infty(\R_+;H^1)}^{-\alpha/2}$. 
By integration, we deduce, as long as $\|u(t)\|_{L^2}$ is not zero,
\begin{equation*}
  \|u(t)\|_{L^2} \le \(\|u_0\|_{L^2}^{\alpha/2} -Cb t\)^{2/\alpha}. 
\end{equation*}
Corollary~\ref{cor:one} then follows. 

\subsection{First part of Theorem~\ref{theo:two}}
\label{sec:two-init}

Suppose now that $d=2$ in Theorem~\ref{theo:exist}. 
In view of \eqref{eq:mass}, we have
\begin{equation*}
   \frac{d}{dt}\|u(t)\|_{L^2}^2 +2b
    \int_M |u(t,x)|^{2-\alpha}dx \le 0 . 
\end{equation*}
Theorem~\ref{theo:exist} and Lemma~\ref{lem:Nash} yield
\begin{equation*}
  \frac{d}{dt}\|u(t)\|_{L^2}^2 +Cb \|u(t)\|_{L^2}^{2}\le 0, 
\end{equation*}
where $C$ is proportional to
$\|u\|_{L^\infty(\R_+;H^1)}^{-\alpha}$. 
By integration, we deduce the first part of Theorem~\ref{theo:two},
that is, the exponential decay of $\|u(t)\|_{L^2(M)}$.

\section{Higher order estimates}
\label{sec:high-order-estim}

As in \cite{CaGa11}, the exponential decay in 2D obtained in the previous
section can be improved to get finite time extinction provided that we
invoke the Nash in equality \eqref{eq:nash2} rather than merely
\eqref{eq:nash1}. This requires of course to control the $H^2$-norm of
$u$.
In order to obtain bounds in $\Sigma^2$, we resume the idea due to
Kato \cite{Kato87} (see also \cite{CazCourant}): to obtain estimates of
order two in space, it suffices to obtain estimates of order one in
time, and to use the equation to relate these quantities. 

\subsection{Evolution of the time derivative}
\label{sec:evol-time-deriv}

Using directly \eqref{eq:main}, for a global weak solution provided by
Theorem~\ref{theo:exist}, we obtain
\begin{align*}
\frac{d}{dt}\|\d_t u\|_{L^2}^2 & = 2\RE \int_M \d_t \bar u \d_t^2
u\\
&= 2\l \IM\int_M \d_t \bar u \d_t\(|u|^{2\si_1}u\) - 2a \RE
\int_M \d_t \bar u \d_t \(|u|^{2\si_2}u\) \\
&\quad -2b\RE \int_M \d_t
\bar u \d_t \(\frac{u}{|u|^{\alpha}}\).   
\end{align*}
For the first term of the right hand side, we use the identity
\begin{equation}\label{eq:imdt}
  \IM\int_M \d_t \bar u \d_t\(|u|^{2\si_1}u\)  =\frac{d}{dt}\(\IM
  \int_M |u|^{2\si_1}u\d_t \bar u\) - \IM\int_M
  |u|^{2\si_1}u\d_t^2 \bar u. 
\end{equation}
The full derivative will be incorporated into the first higher energy,
so we focus on the last term. From the equation,
\begin{align*}
   - \IM\int_M
  |u|^{2\si_1}u\d_t^2 \bar u & = - \IM\int_M
  |u|^{2\si_1}u \d_t \( -\frac{i}{2}\Delta \bar u +iV \bar u +i \l
  |u|^{2\si_1} \bar u -a  |u|^{2\si_2}\bar u -b
  \frac{\bar u}{|u|^{\alpha}}\)\\
& = \frac{1}{2}\RE\int_M |u|^{2\si_1} u\d_t \Delta \bar u
-\frac{1}{2(\si_1+1)}\frac{d}{dt}\int_M  V|u|^{2\si_1+2}\\
&\quad -\frac{\l}{2}\frac{d}{dt}\int_M |u|^{4\si_1+2} +a\IM \int_M
|u|^{2\si_1}u \d_t\(|u|^{2\si_2}\bar u\) \\
&\quad +b \IM \int_M
|u|^{2\si_1} u \d_t \( \frac{\bar u}{|u|^{\alpha}}\)
\end{align*}
For the first term, we invoke the fact that the Laplacian is
self-adjoint,and use the identity
\begin{equation*}
  \Delta \(|u|^{2\si_1} u\) = u\Delta|u|^{2\si_1}  +2\nabla u\cdot
  \nabla |u|^{2\si_1} +|u|^{2\si_1}\Delta u.
\end{equation*}
We also compute
\begin{equation*}
  \Delta|u|^{2\si_1} = \si_1(\si_1-1)|u|^{2\si_1-4}\left|
    \nabla|u|^2\right|^2 +2\si_1|u|^{2\si_1-2}\(|\nabla u|^2
  +\RE\(\bar u \Delta u\)\). 
\end{equation*}
Therefore,
\begin{align*}
  \frac{1}{2}\RE \int_Mu \d_t \bar u  \Delta |u|^{2\si_1} & =
  \frac{ \si_1(\si_1-1)}{2} \RE \int_M u\d_t \bar u |u|^{2\si_1-4}\left|
    \nabla|u|^2\right|^2 \\
&\quad +\si_1 \RE \int_M u\d_t \bar u \(|\nabla u|^2
  +\RE\(\bar u \Delta u\)\)|u|^{2\si_1-2}.
\end{align*}
The first two terms can be factored out in a more concise way in order
to emphasize an exact time derivative:
\begin{equation*}
  \frac{ \si_1(\si_1-1)}{2} \RE \int_M u\d_t \bar u |u|^{2\si_1-4}\left|
    \nabla|u|^2\right|^2  = \frac{\si_1}{4}\int_M \d_t |u|^{2\si_1-2} \left|
    \nabla|u|^2\right|^2 ,
\end{equation*}
and
\begin{equation*}
  \si_1 \RE \int_M u\d_t \bar u |\nabla u|^2 |u|^{2\si_1-2} =
  \frac{1}{2}\int_M \d_t |u|^{2\si_1}  |\nabla u|^2 . 
\end{equation*}
We compute $\RE(\bar u \Delta u)$ by using \eqref{eq:main}:
\begin{equation*}
  \RE(\bar u \Delta u)= 2\IM \(\bar u \d_t u \) +2V |u|^2 +2\l
  |u|^{2\si_1+2}, 
\end{equation*}
and we end up with
\begin{align*}
   \frac{1}{2}\RE \int_M u\d_t \bar u  \Delta |u|^{2\si_1} & =
   \frac{\si_1}{4}\int_M \d_t |u|^{2\si_1-2} \left| 
    \nabla|u|^2\right|^2 + \frac{1}{2}\int_M \d_t |u|^{2\si_1}
  |\nabla u|^2 \\
&\quad +\int_M \d_t|u|^{2\si_1} \IM \(\bar u \d_t u\) \\
& \quad +\frac{\si_1}{\si_1+1}\frac{d}{dt}\int_M V|u|^{2\si_1+2}
+\frac{\l \si_1}{2\si_1+1}\frac{d}{dt}\int_M |u|^{4\si_1+2}. 
\end{align*}
We also note that
\begin{equation*}
  \int_M \d_t|u|^{2\si_1} \IM \(\bar u \d_t u\)  = \IM \int_M \d_t u
  \d_t \(|u|^{2\si_1} \bar u\) = - \IM \int_M \d_t \bar u
  \d_t \(|u|^{2\si_1} u\),
\end{equation*}
so that we recover the left hand side of \eqref{eq:imdt}, with the
opposite sign. Therefore, we have
\begin{align*}
  2 \IM \int_M \d_t \bar u
  \d_t \(|u|^{2\si_1} u\) & =\frac{d}{dt}\( \IM
  \int_M |u|^{2\si_1}u\d_t \bar u +\frac{2\si_1-1}{2\si_1+2}\int_M
  V|u|^{2\si_1+2}\)\\
&\quad  -\frac{d}{dt}\(\frac{\l}{4\si_1+2}\int_M |u|^{4\si_1+2}\)\\
&\quad +  \frac{\si_1}{4}\int_M \d_t |u|^{2\si_1-2} \left| 
    \nabla|u|^2\right|^2 + \frac{1}{2}\int_M \d_t |u|^{2\si_1}
  |\nabla u|^2 \\
&\quad + \RE \int_M \d_t \bar u \nabla u\cdot \nabla |u|^{2\si_1}
+\frac{1}{2}\RE \int_M |u|^{2\si_1}\d_t \bar u \Delta u. 
\end{align*}
For the last term, we use \eqref{eq:main} to substitute $\Delta u$:
\begin{align*}
  \frac{1}{2}\RE \int_M |u|^{2\si_1}\d_t \bar u \Delta u &= 
\RE \int_M  |u|^{2\si_1}\d_t \bar u \( Vu +\l |u|^{2\si_1}u -ia
|u|^{2\si_2}u -ib \frac{u}{|u|^{\alpha}}\)\\
&= \frac{1}{2\si_1+2}\frac{d}{dt}\int_M V |u|^{2\si_1+2}+
\frac{\l}{4\si_1 +2}\frac{d}{dt}\int_M |u|^{4\si_1+2}\\
&\quad +a \IM \int_M |u|^{2\si_1+2\si_2} u\d_t \bar u + b \int_M
|u|^{2\si_1-\alpha} \IM \( u\d_t \bar u\). 
\end{align*}
At this stage, we infer
\begin{align*}
  \frac{d}{dt}& \( \|\d_t u\|_{L^2}^2 -\l \int_M |u|^{2\si_1}\IM \(u\d_t
  \bar u\) -\frac{\l \si_1}{\si_1+1}\int_M V |u|^{2\si_1+2}\)=\\
& \frac{\l \si_1}{4}\int_M \d_t |u|^{2\si_1-2}  \left| 
    \nabla|u|^2\right|^2 + \frac{\l}{2}\int_M \d_t |u|^{2\si_1}
  |\nabla u|^2+\l \RE \int_M \d_t \bar u \nabla u\cdot \nabla
  |u|^{2\si_1}\\
+ & \l a \int_M |u|^{2\si_1+2\si_2}\IM\( u \d_t \bar u\) +\l b \int_M
|u|^{2\si_1-\alpha}\IM\( u \d_t \bar u\)  \\
-& 2a \RE \int_M \d_t \bar u \d_t \( |u|^{2\si_2}u \) -2b \RE \int_M
\d_t \bar u \d_t \(\frac{u}{|u|^\alpha}\). 
\end{align*}
The final simplification consists in developing the last two terms in
the following fashion:
\begin{align*}
  \RE \int_M \d_t \bar u \d_t \( |u|^p u \) &= \(\frac{p}{2}+1\)\int_M
  |u|^{p} |\d_t u|^2 +\frac{p}{2} \int_M |u|^{p-2} \RE \(u \d_t
  \bar u\)^2 \\
 &= \(\frac{p}{2}+1\)\int_M
  |u|^{p} |\d_t u|^2 \\
&\quad +\frac{p}{2} \int_M |u|^{p-2}\( \(\RE u \d_t
  \bar u\)^2 - \(\IM u\d_t \bar u \)^2\)\\
&= \(\frac{p}{2}+1\)\int_M
  |u|^{p-2} \( \(\RE u \d_t
  \bar u\)^2 + \(\IM u\d_t \bar u \)^2\) \\
&\quad +\frac{p}{2} \int_M |u|^{p-2}\( \(\RE u \d_t
  \bar u\)^2 - \(\IM u\d_t \bar u \)^2\)\\
&= \(p+1\)\int_M
  |u|^{p-2}  \(\RE u \d_t
  \bar u\)^2  + \int_M |u|^{p-2} \(\IM u\d_t \bar u \)^2.
\end{align*}
We conclude:
\begin{proposition}\label{prop:dtu}
 Let $u_0\in\Sigma^2$. In either of the cases considered in
 Theorem~\ref{theo:exist}, the global weak solution $u$ satisfies:
\begin{align*}
  \frac{d}{dt}&\(\|\d_t u\|_{L^2}^2 -\l \IM\int_M |u|^{2\si_1}u\d_t
  \bar u -\frac{\l \si_1}{\si_1+1}\int_M V |u|^{2\si_1+2}\)=\\
&\frac{\l\si_1}{4}\int_M\d_t|u|^{2\si_1-2}\left|\nabla
  |u|^2\right|^2 +\frac{\l}{2}\int_M\d_t|u|^{2\si_1} |\nabla
u|^2+\l \RE \int_M \d_t \bar u \nabla u \cdot \nabla |u|^{2\si_1}\\ 
+&\l
a \int_M|u|^{2\si_1+2\si_2}\IM \(u\d_t \bar u\) +\l b \int_M
|u|^{2\si_1-\alpha}\IM\(u\d_t \bar u\) \\
-&2a (2\si_2+1)\int_M
|u|^{2\si_2-2}\(\RE u\d_t \bar u\)^2
- 2a\int_M
|u|^{2\si_2-2}\(\IM u\d_t \bar u\)^2\\
-&2b(1-\alpha)\int_M
|u|^{-2-\alpha}\(\RE u\d_t \bar u\)^2 -2b\int_M
|u|^{-2-\alpha}\(\IM u\d_t \bar u\)^2 .
\end{align*}
\end{proposition}

\subsection{From order one in time to order two in space}
\label{sec:from-order-one}

We rewrite the quantity involved in Proposition~\ref{prop:dtu} in
order to get rid of all time derivatives:
\begin{align*}
 \|\d_t u\|_{L^2}^2 -\l \IM\int |u|^{2\si_1}u\d_t
  \bar u -\frac{\l \si_1}{\si_1+1}\int  V |u|^{2\si_1+2}& =
 \RE\int \( \d_t u +i\l |u|^{2\si_1}u\)\d_t \bar u  \\
&\quad -\frac{\l
  \si_1}{\si_1+1}\int V |u|^{2\si_1+2} 
\end{align*}
Leaving out the real part for one moment, the first integral on the
right hand side is rewritten as
\begin{align*}
 \int \( \frac{i}{2}\Delta u -iV u -a |u|^{2\si_2}u
-b\frac{u}{|u|^\alpha}\)\(-\frac{i}{2} \Delta \bar u +iV\bar u +i\l
|u|^{2\si_1}\bar u -a |u|^{2\si_2}\bar u -b \frac{\bar
  u}{|u|^\alpha}\),
\end{align*}
whose real part is equal to:
\begin{align*}
  &\frac{1}{4}\|\Delta u\|_{L^2}^2 -\RE \int V\bar
    u \Delta u - \frac{\l}{2}\RE \int |u|^{2\si_1}\bar u \Delta u +a
    \IM \int |u|^{2\si_2}\bar u \Delta u \\
&+ b\IM\int \frac{\bar
      u}{|u|^{\alpha}} \Delta u +\int
    V^2|u|^2 +\frac{\l}{\si_1+1}\int V |u|^{2\si_1+2} +a^2 \int
    |U|^{4\si_2+2} \\
&+2ab\int |u|^{2\si_2+2-\alpha} + b^2\int
    |u|^{2-2\alpha}. 
\end{align*}
Note that it is in order for the last term to belong to some
reasonable Lebesgue space that we assume $\alpha\le 1/2$ in the case
where $M=\R^2$. By integration by parts, we can also write
\begin{align*}
  -\RE \int V\bar u \Delta u &= \int V|\nabla u|^2 -\frac{1}{2}\int
  |u|^2 \Delta V,\\
 - \frac{\l}{2}\RE \int |u|^{2\si_1}\bar u \Delta u &=
\l\frac{\si_1+1}{2}\int |u|^{2\si_1}|\nabla u|^2 +\l
\frac{\si_1}{2}\RE \int |u|^{2\si_1-2}\bar u ^2 (\nabla u)^2,\\
 a \IM \int |u|^{2\si_2}\bar u \Delta u & = -a\si_2 \IM \int
|u|^{2\si_2-2}\bar u^2 (\nabla u)^2. 
\end{align*}
Gathering all the terms together, this leads us to setting as a second
order energy:
\begin{equation*}
\begin{aligned}
  {\mathcal E}_2(t)&:= \frac{1}{4}\|\Delta u\|_{L^2}^2 + \int
    V^2|u|^2 + \int V|\nabla u|^2 \\
&\quad +a^2 \int
    |u|^{4\si_2+2}+2ab\int |u|^{2\si_2+2-\alpha} + b^2\int
    |u|^{2-2\alpha}\\
&\quad -\frac{1}{2}\int  |u|^2 \Delta V +\l\frac{\si_1+1}{2}\int
|u|^{2\si_1}|\nabla u|^2  +\l
\frac{\si_1}{2}\RE \int |u|^{2\si_1-2}\bar u ^2 (\nabla u)^2 \\
&\quad  -a\si_2 \IM \int
|u|^{2\si_2-2}\bar u^2 (\nabla u)^2 + b\IM\int \frac{\bar
      u}{|u|^{\alpha}} \Delta u  +\frac{\l}{\si_1+1}\int V
    |u|^{2\si_1+2}. 
\end{aligned}
\end{equation*}
\begin{lemma}\label{lem:E2}
Let $u$ be given by Theorem~\ref{theo:exist}. 
\begin{itemize}
\item There exists $C>0$
  such that for all $t\ge 0$, 
  \begin{equation*}
    \frac{1}{C} \|u(t)\|_{\Sigma^2}^2 \le  {\mathcal E}_2(t) \le C
  \|u(t)\|_{\Sigma^2}^2+C. 
  \end{equation*}
\item There exists $C$ such that for all $t\ge 0$, $\|\d_t
  u(t)\|_{L^2}^2 \le C  {\mathcal E}_2(t)$. 
\end{itemize}
\end{lemma}
\begin{proof}
  The first two terms in  $\mathcal
  E_2$ correspond to the definition of $\|u(t)\|_{\Sigma^2}^2$, up to
  irrelevant multiplying constants. The third term is non-negative,
  and is controlled by  $\|u(t)\|_{\Sigma^2}^2$, as shown by an
  integration by parts. 

The three terms on the second line in the definition of $\mathcal
E_2$  are non-negative. The first two terms are controlled  by some power of
$\|u(t)\|_{H^1}$, which is 
uniformly bounded from Theorem~\ref{theo:exist}. The third term is
controlled by the $L^2$-norm of $u$ if $M$ is compact, thanks to
Lemma~\ref{lem:GNdual} if $M=\R^2$ and $\alpha<1/2$. If $M=\R^2$ and
$\alpha=1/2$, it is easy to check that
\begin{equation}\label{eq:GNdual2}
  \|f\|_{L^1(\R^2)}\le
  C\|f\|_{L^2(\R^2)}^{1/2}\||x|^2f\|_{L^2(\R^2)}^{1/2},\quad \forall f\in
  \Sigma^2. 
\end{equation}
Since $\Delta V$ is bounded, $\int |u|^2 \Delta V$ is equivalent to
$\|u\|_{L^2}^2$. The last two terms on the third line are both
controlled as follows: for $0<\eps<1$,
\begin{equation*}
  \int |u|^{2\si_1}|\nabla u|^2 \le \|u\|_{L^{2/\eps}}^{2\si_1}
  \|\nabla u\|_{L^{2/(1-\eps)}}^2\lesssim \|u\|_{H^1}^{2\si_1}\|\nabla
  u\|_{L^2}^{2(1-\eps)} \|\Delta u\|_{L^2}^{2\eps},
\end{equation*}
where we have used Gagliardo-Nirenberg inequality (Lemma~\ref{lem:GN})
applied to $\nabla u$ for the last inequality. The first term of the
fourth line is controlled in exactly the same fashion, by simply
replacing $\si_1$ with $\si_2$. 

By Cauchy-Schwarz inequality, we have
\begin{equation*}
 \left| \IM\int \frac{\bar  u}{|u|^{\alpha}} \Delta u \right| \le
 \|u\|_{L^{2-2\alpha}}^{1-\alpha} \|\Delta u\|_{L^2}. 
\end{equation*}
If $M$ is compact, we conclude by H\"older inequality,
\begin{equation*}
  \|u\|_{L^{2-2\alpha}} \le |M|^{\alpha/(4-2\alpha)}\|u\|_{L^2}.
\end{equation*}
If $M=\R^2$, we proceed as above, by either invoking
Lemma~\ref{lem:GNdual} if $\alpha<1/2$, or \eqref{eq:GNdual2} if
$\alpha=1/2$. 

Finally, Cauchy-Schwarz inequality and Sobolev embedding yield.
\begin{equation*}
  \int V
    |u|^{2\si_1+2} \le \|V u\|_{L^2}
    \|u\|_{L^{4\si_2+2}}^{2\si_2+1}\le C \|u\|_{\Sigma}^{2\si_1+2},
\end{equation*}
hence the first point of the lemma. 
\smallbreak

For the second point, recall that we also have, by construction,
\begin{equation*}
  \mathcal E_2(t) = \|\d_t u\|_{L^2}^2 -\l \IM\int |u|^{2\si_1}u\d_t
  \bar u -\frac{\l \si_1}{\si_1+1}\int  V |u|^{2\si_1+2}.
\end{equation*}
We have just seen that the last term is estimated as
\begin{equation*}
  \int V
    |u|^{2\si_1+2} \lesssim \|u\|_{\Sigma}^{2\si_1+2}.
\end{equation*}
For the second term, Cauchy-Schwarz inequality, Sobolev embedding and
Young inequality yield
\begin{align*}
 |\l | \left| \IM\int |u|^{2\si_1}u\d_t
  \bar u\right| &\le |\l  | \|\d_t
u\|_{L^2}\|u\|_{L^{4\si1+2}}^{2\si_1+1}\lesssim \|\d_t
u\|_{L^2}\|u\|_{\Sigma}^{2\si_1+1}\\
&\le \eps \|\d_t
u\|_{L^2}^2 + \frac{C}{\eps} \|u\|_{\Sigma}^{4\si_1+2},
\end{align*}
hence the second point of the lemma by choosing $\eps=1/2$.
\end{proof}

\section{Finite time extinction in 2D}
\label{sec:two-end}
We recall the celebrated Br\'ezis-Gallou\"et inequality, established
in \cite{BrGa80}. 
\begin{lemma}[Br\'ezis-Gallou\"et inequality]\label{lem:brezis-gallouet}
  Let $d=2$ in Assumption~\ref{hyp:main}. There exists $C$ such that
  for all $f\in H^2(M)$, 
  \begin{equation*}
    \|f\|_{L^\infty(M)}\le C\(\|f\|_{H^1(M)} \sqrt{ \ln \(2+
      \|f\|_{H^2(M)} \)}+1\). 
  \end{equation*}
\end{lemma}

Recall that by construction, the time derivative of $\mathcal E_2$ is
given by Proposition~\ref{prop:dtu}. Since the last two lines are
non-negative, and noticing that all the terms in the second line can
be estimated in a common fashion, we have:
\begin{equation}\label{eq:dotE2}
  \dot {\mathcal E}_2 \lesssim \int |u|^{2\si_1-1} |\d_t u| |\nabla
  u|^2 + \int |u|^{2\si_1+2\si_2 +1}|\d_t u| + \int
  |u|^{2\si_1-\alpha+1}|\d_t u |.  
\end{equation}
The first term is controlled, up to a multiplicative constant, by 
\begin{equation*}
  \|u\|_{L^\infty}^{2\si_1-1} \|\nabla u\|_{L^4}^2\|\d_t
  u\|_{L^2}\lesssim \|u\|_{L^\infty}^{2\si_1-1} \|\Delta u\|_{L^2} \|\d_t
  u\|_{L^2},
\end{equation*}
where we have used Gagliardo-Nirenberg inequality. Using
Lemma~\ref{lem:E2}, we infer
\begin{equation*}
  \int |u|^{2\si_1-1} |\d_t u| |\nabla
  u|^2 \lesssim \|u\|_{L^\infty}^{2\si_1-1} {\mathcal E}_2. 
\end{equation*}
Br\'ezis-Gallou{\"e}t inequality implies:
\begin{equation}\label{eq:step1}
   \int |u|^{2\si_1-1} |\d_t u| |\nabla
  u|^2 \lesssim \(\|u\|_{\Sigma}\sqrt{\ln (2+\mathcal E_2)}+1\)^{2\si_1-1} {\mathcal E}_2. 
\end{equation}
The last two terms in \eqref{eq:dotE2} are estimated thanks to
Cauchy-Schwarz inequality, Sobolev embedding and the second point in
Lemma~\ref{lem:E2}: 
\begin{equation*}
  \int |u|^{2\si_1+2\si_2 +1}|\d_t u| + \int
  |u|^{2\si_1-\alpha+1}|\d_t u | \lesssim \(
  \|u\|_\Sigma^{2\si_1+2\si_2 +1} + \|u\|_\Sigma^{2\si_1
    +1-\alpha}\)\mathcal E_2.  
\end{equation*}
Along with \eqref{eq:step1}, \eqref{eq:dotE2} then yields
\begin{equation*}
   \dot {\mathcal E}_2 \le K(\|u\|_{\Sigma})\( 1+ \ln (2+\mathcal
   E_2)\)^{\si_1-1/2} \( 2+  {\mathcal E}_2 \),
\end{equation*}
where $K(\cdot)$ denotes a continuous function. Integrating in time, we infer that 
\begin{equation*}
F(t):=
\left\{
\begin{aligned}
\(  1+ \ln (2+\mathcal
   E_2(t))\)^{3/2-\si_1}&\quad \text{if}\quad\si_1<\frac{3}{2},\\
\ln \(  1+ \ln (2+\mathcal
   E_2(t))\)&\quad \text{if}\quad\si_1=\frac{3}{2},
\end{aligned}
\right.
\end{equation*}
is controlled by $F(0) + t K(\|u_0\|_{\Sigma})$, where we have used also
Proposition~\ref{prop:apriori} (after passing to the limit $\delta\to
0$), up to changing the continuous function $K$. In order to ease
notations, we now denote by $K_j$ any positive continuous 
function of $\|u_0\|_{\Sigma^j}$, which may change from line to line,
but only finitely many times. 

\smallbreak
\noindent {\bf Case $\si_1<3/2$.} In this case, the control on $F$
yields, along with Lemma~\ref{lem:E2},
\begin{equation*}
  \|u(t)\|_{H^2} \le K_2 e^{t^\frac{2}{3-2\si_1}
    K_1}. 
\end{equation*}
Nash inequality \eqref{eq:nash2} then implies
\begin{equation*}
  \|u(t)\|_{L^2}\le K_2
  \|u(t)\|_{L^{2-\alpha}}^{2(2-\alpha)/(4-\alpha)}
  e^{t^\frac{2}{3-2\si_1}   K_1}.  
\end{equation*}
Let $\theta=1-\alpha/4$. The above inequality and \eqref{eq:mass}
yield
\begin{equation*}
  \frac{d}{dt}\|u(t)\|_{L^2}^2 \le
  -2\|u(t)\|_{L^{2-\alpha}}^{2-\alpha} \le -  K_2
  e^{-t^\frac{2}{3-2\si_1}   K_1}\|u(t)\|_{L^2}^{2\theta},
\end{equation*}
hence
\begin{equation*}
   \frac{d}{dt}\|u(t)\|_{L^2}^{2(1-\theta)} \le -  K_2
  e^{-t^\frac{2}{3-2\si_1}   K_1}. 
\end{equation*}
By integration, we infer
\begin{equation}\label{eq:3pm}
  \|u(t)\|_{L^2}^{2(1-\theta)} \le \|u_0\|_{L^2}^{2(1-\theta)} -
  K_2 \int_0^t   e^{-\tau^\frac{2}{3-2\si_1}  K_1}d\tau. 
\end{equation}
By changing variables in the integral, note that there exists a
${\bf K_2}=K(\|u_0\|_{\Sigma^2})$ such that 
\begin{equation*}
   K_2 \int_0^\infty   e^{-\tau^\frac{2}{3-2\si_1}
     K_1}d\tau
 ={\bf K_2} \int_0^\infty
 e^{-\tau^\frac{2}{3-2\si_1}}d\tau,
\end{equation*}
where the integrals are obviously finite, and the last one is
independent of $u_0$. We conclude that if
\begin{equation}\label{eq:CIsmall}
  \|u_0\|_{L^2}^{2(1-\theta)} -{\bf K_2} \int_0^\infty
 e^{-\tau^\frac{2}{3-2\si_1}}d\tau<0,
\end{equation}
then there for $t$ sufficiently large, the right hand side in
\eqref{eq:3pm} becomes zero. Therefore, there exists some finite time $T>0$ such that
$\|u(T)\|_{L^2}=0$. Since \eqref{eq:CIsmall} corresponds to a
smallness assumption on $\|u_0\|_{L^2}$ when $\|u_0\|_{\Sigma^2}$ is
fixed, the second point in  Theorem~\ref{theo:two} follows in the case
$1/2\le \si_1<3/2$.  
\smallbreak
\noindent {\bf Case $\si_1=3/2$.} The control on $F$ now leads to a
control by a double exponential:
\begin{equation*}
  \|u(t)\|_{H^2}\le  \exp \(K_2 e^{K_1 t}\). 
\end{equation*}
In the same fashion as above, we infer
\begin{equation*}
   \frac{d}{dt}\|u(t)\|_{L^2}^{2(1-\theta)} \le - \exp \(-K_2 e^{K_1 t}\),
\end{equation*}
hence
\begin{equation*}
  \|u(t)\|_{L^2}^{2(1-\theta)} \le \|u_0\|_{L^2}^{2(1-\theta)} -
   \int_0^t  \exp \(-K_2 e^{K_1 \tau}\) d\tau. 
\end{equation*}
We have
\begin{align*}
  \int_0^\infty  \exp \(-K_2 e^{K_1 \tau}\) d\tau =
  \frac{1}{K_1}\int_0^\infty \exp \(-K_2 e^{\tau}\) d\tau& = 
\frac{1}{K_1}\int_{\ln K_2}^\infty \exp \( -e^{\tau}\) d\tau\\
&\ge
\frac{1}{K_1(R)}\int_{\ln K_2(R)}^\infty \exp \( -e^{\tau}\) d\tau,
\end{align*}
that is, a constant which depends only on $R$ provided that
$\|u_0\|_{\Sigma^2}\le R$. Finite time extinction then follows as soon
as
\begin{equation*}
  \|u_0\|_{L^2}^{2(1-\theta)} < \frac{1}{K_1(R)}\int_{\ln
    K_2(R)}^\infty \exp \( -e^{\tau}\) d\tau. 
\end{equation*}

 \bibliographystyle{siam}
\bibliography{biblio}

\begin{thebibliography}{10}

\bibitem{AdAtCa06}
{\sc S.~Adly, H.~Attouch, and A.~Cabot}, {\em Finite time stabilization of
  nonlinear oscillators subject to dry friction}, in Nonsmooth mechanics and
  analysis, vol.~12 of Adv. Mech. Math., Springer, New York, 2006,
  pp.~289--304.

\bibitem{Ag82}
{\sc S.~Agmon}, {\em Lectures on exponential decay of solutions of second-order
  elliptic equations: bounds on eigenfunctions of {$N$}-body {S}chr\"odinger
  operators}, vol.~29 of Mathematical Notes, Princeton University Press,
  Princeton, NJ; University of Tokyo Press, Tokyo, 1982.

\bibitem{AnCaSp-p}
{\sc P.~Antonelli, R.~Carles, and C.~Sparber}, {\em On nonlinear
  {S}chr{\"o}dinger type equations with nonlinear damping}, Int. Math. Res.
  Not.
\newblock To appear.

\bibitem{AnMa09}
{\sc P.~Antonelli and P.~Marcati}, {\em On the finite energy weak solutions to
  a system in quantum fluid dynamics}, Comm. Math. Phys., 287 (2009),
  pp.~657--686.

\bibitem{AnSp10}
{\sc P.~Antonelli and C.~Sparber}, {\em Global well-posedness for cubic {NLS}
  with nonlinear damping}, Comm. Partial Differential Equations, 35 (2010),
  pp.~2310--2328.

\bibitem{BrGa80}
{\sc H.~Br{\'e}zis and T.~Gallouet}, {\em Nonlinear {S}chr\"odinger evolution
  equations}, Nonlinear Anal., 4 (1980), pp.~677--681.

\bibitem{BGT}
{\sc N.~Burq, P.~G\'erard, and N.~Tzvetkov}, {\em Strichartz inequalities and
  the nonlinear {S}chr\"odinger equation on compact manifolds}, Amer. J. Math.,
  126 (2004), pp.~569--605.

\bibitem{CaDaSa12}
{\sc R.~Carles, R.~Danchin, and J.-C. Saut}, {\em Madelung,
  {G}ross-{P}itaevskii and {K}orteweg}, Nonlinearity, 25 (2012),
  pp.~2843--2873.

\bibitem{CaGa11}
{\sc R.~Carles and C.~Gallo}, {\em Finite time extinction by nonlinear damping
  for the {S}chr{\"o}dinger equation}, Comm. Part. Diff. Eq., 36 (2011),
  pp.~961--975.

\bibitem{CaMi04}
{\sc R.~Carles and L.~Miller}, {\em Semiclassical nonlinear {S}chr\"odinger
  equations with potential and focusing initial data}, Osaka J. Math., 41
  (2004), pp.~693--725.

\bibitem{CazCourant}
{\sc T.~Cazenave}, {\em Semilinear {S}chr\"odinger equations}, vol.~10 of
  Courant Lecture Notes in Mathematics, New York University Courant Institute
  of Mathematical Sciences, New York, 2003.

\bibitem{Fujiwara}
{\sc D.~Fujiwara}, {\em Remarks on the convergence of the {F}eynman path
  integrals}, Duke Math. J., 47 (1980), pp.~559--600.

\bibitem{GV85c}
{\sc J.~Ginibre and G.~Velo}, {\em The global {C}auchy problem for the
  nonlinear {S}chr\"odinger equation revisited}, Ann. Inst. H. Poincar\'e Anal.
  Non Lin\'eaire, 2 (1985), pp.~309--327.

\bibitem{Yu63}
{\sc V.~I. Judovi{\v{c}}}, {\em Non-stationary flows of an ideal incompressible
  fluid}, \u Z. Vy\v cisl. Mat. i Mat. Fiz., 3 (1963), pp.~1032--1066.

\bibitem{Ka72}
{\sc T.~Kato}, {\em Schr\"odinger operators with singular potentials}, Israel
  J. Math., 13 (1972), pp.~135--148.

\bibitem{Kato87}
{\sc T.~Kato}, {\em On nonlinear {S}chr\"odinger equations}, Ann. IHP (Phys.
  Th\'eor.), 46 (1987), pp.~113--129.

\bibitem{KavianWeissler}
{\sc O.~Kavian and F.~Weissler}, {\em Self-similar solutions of the
  pseudo-conformally invariant nonlinear {S}chr\"odinger equation}, Michigan
  Math. J., 41 (1994), pp.~151--173.

\bibitem{Na58}
{\sc J.~Nash}, {\em Continuity of solutions of parabolic and elliptic
  equations}, Amer. J. Math., 80 (1958), pp.~931--954.

\bibitem{OgOz91}
{\sc T.~Ogawa and T.~Ozawa}, {\em Trudinger type inequalities and uniqueness of
  weak solutions for the nonlinear {S}chr\"odinger mixed problem}, J. Math.
  Anal. Appl., 155 (1991), pp.~531--540.

\bibitem{Oz95}
{\sc T.~Ozawa}, {\em On critical cases of {S}obolev's inequalities}, J. Funct.
  Anal., 127 (1995), pp.~259--269.

\bibitem{OzVi-p}
{\sc T.~Ozawa and N.~Visciglia}, {\em An improvement on the
  {B}r\'ezis-{G}allou{\"e}t technique for {2D} {NLS} and {1D} half-wave
  equation}.
\newblock Preprint, archived as \url{http://arxiv.org/abs/1403.7443}.

\bibitem{Taylor3}
{\sc M.~Taylor}, {\em Partial differential equations. {III}}, vol.~117 of
  Applied Mathematical Sciences, Springer-Verlag, New York, 1997.
\newblock Nonlinear equations.

\end{thebibliography}

\end{document}